\documentclass[a4paper, 12pt,oneside,reqno]{amsart}
\usepackage[a4paper]{geometry}
\usepackage{amssymb,amsmath,amsthm}
\usepackage[T2A]{fontenc}
\usepackage[breaklinks=true]{hyperref}

\usepackage{xcolor}

\usepackage{bbm}
\usepackage{graphicx}

\def\Var{\mathrm{Var}}
\def\dim{\mathrm{dim}}
\def\Der{\mathrm{Der}}
\def\Nov{\mathrm{Nov}}
\def\Com{\mathrm{Com}}
\def\As{\mathrm{As}}
\def\Bi{\mathrm{Bi}}

\def\Zinb{\mathrm{Zinb}}

\def\wt{\mathrm{wt}}

\def\Der {\mathop {\fam 0 Der}\nolimits}

\let\<\langle
\let\>\rangle

\newtheorem{definition}{Definition}

\newtheorem{proposition}{Proposition}
\newtheorem{theorem}{Theorem}
\newtheorem{corollary}{Corollary}

\title{On the differential Novikov algebras}

\author{A. Dauletiyarova}

\address{Institute of Mathematics and Mathematical Modeling, Almaty, Kazakhstan}

\address{Narxoz University, Almaty, Kazakhstan}

\email{d\_aigera95@mail.ru}

\author{B. Sartayev$^{*}$}

\address{Narxoz University, Almaty, Kazakhstan}

\address{SDU University, Kaskelen, Kazakhstan}

\email{baurjai@gmail.com}



\thanks{${}^{*}$Corresponding author: Bauyrzhan Sartayev   (baurjai@gmail.com)}

\begin{document}

\begin{abstract}
We study the quadratic operad $\Der\Nov_q$ determined by the
degree-$3$ identities arising from the operad $\Nov\circ\Nov$,
and its quadratic dual $\Der\Nov^!$. We show that the natural
epimorphism
\[
\Der\Nov_q\longrightarrow \Nov\circ\Nov
\]
is not an isomorphism: in arity $4$ the corresponding dimensions
are $491$ and $400$, respectively. We obtain a presentation of
$\Der\Nov^!$ by two binary operations and show that every mixed
monomial can be reduced to a linear combination of pure monomials.
The pure $\prec$-component is described as a free right Novikov
algebra subject to three additional identities, while the pure
$\succ$-component is a free bicommutative algebra subject to two
additional identities. We construct explicit bases for both
components, determine the dimensions of the multilinear components,
and prove that the bicommutative operad and $\Der\Nov^!$ satisfy
the Dong property.
\end{abstract}

\maketitle

\section{Introduction}

In recent years, Novikov algebras have led to several notable results in combinatorial algebra. For example, the Specht property for Novikov algebras was established in \cite{DotIsmUmi2023}. A fundamental combinatorial result for free Novikov algebras is the monomial basis of $\Nov\<X\>$ constructed in \cite{DzhLofwall}. It was also shown that every Novikov algebra can be embedded into an appropriate commutative algebra with a derivation \cite{BCZ2017}. The analogue for noncommutative Novikov algebras was obtained in \cite{erlagol2021}: every noncommutative Novikov algebra can be embedded into an appropriate associative algebra with a derivation.

The defining identities of the variety of (left-)Novikov algebras first appeared in \cite{GD79}. That work also contains an early construction of a Novikov product on a differential commutative associative algebra, namely
$$a\circ b=ad(b),$$
which satisfies the following identities:
\begin{equation}\label{right-com}
    (a\circ b)\circ c=(a\circ c)\circ b
\end{equation}
\begin{equation}\label{left-sym}
    (a\circ b)\circ c-a\circ (b\circ c)=(b\circ a)\circ c-b\circ (a\circ c).
\end{equation}

The motivation for this work comes from the white and black Manin
products of the Novikov operad with itself. Since $\Nov=\Nov^!$,
these constructions are naturally related by quadratic duality.
In the present situation, however, the degree-$3$ relations arising
from $\Nov\circ\Nov$ do not give a complete presentation of the
whole operad. We therefore distinguish the quadratic operad
$\Der\Nov_q$, determined by these relations, from
$\Der\Nov=\Nov\circ\Nov$. The main object studied below is the
quadratic dual $\Der\Nov^!$.

Let $\Der\Nov$ be the operad obtained from $\Nov\circ\Nov$, and let $\Der\Nov\<X\>$ be the free algebra of the variety $\Der\Nov$ generated by a set $X$.
A useful description of $\Der\Nov\<X\>$ is given by the following chain of inclusions:
\begin{equation}\label{inclusion}
\Der\Nov\langle X\rangle \subset 
\Nov\langle X, \partial \rangle = \Nov \langle X^{(\omega )}\rangle 
\subset \Com\langle X^{(\omega )}, d \rangle = \Com \langle X^{(\omega , \omega )} \rangle .    
\end{equation}
Here $X^{(\omega ,\omega )} = (X^{(\omega )})^{(\omega )} = \{x^{(n,m)} \mid 
x\in X, n,m\in \mathbb Z_+ \}$, and the variable $x^{(n,m)}$ represents the generator obtained by applying the two commuting derivations $d$ and $\partial$, that is, $d^n\partial^m(x)$. For more details on (\ref{inclusion}), see \cite{KolMashSar}. The elements of $\Der\Nov \langle X\rangle$
are precisely the polynomials in $\Com \langle X^{(\omega , \omega )} \rangle$
that can be written as linear combinations of monomials 
\[
x_1^{(n_1,m_1)}\dots x_k^{(n_k,m_k)},\quad \sum_i n_i = \sum_i m_i = k+1.
\]

We briefly recall several related results on Novikov algebras. A basis of the free metabelian Novikov algebra was obtained in \cite{MetNov}, and a basis of the free noncommutative Novikov algebra was constructed in \cite{DauSar}. Classifications of simple Novikov algebras were studied in \cite{Xu1996,Xu2001,Zelm}, while examples of finite-dimensional Novikov algebras can be found in \cite{Ivan1, Ivan2}. Various structural questions were considered in \cite{Osborn92,Osborn92-1,Osborn94}; more recently, prime and semiprime Novikov algebras were studied in \cite{Pan2022}.

Since varieties defined by polynomial identities of degrees $2$ and $3$ are in one-to-one correspondence with quadratic operads, we will use the same terminology for a variety $\Var$ and for the corresponding operad throughout the paper. All algebras are considered over a field $\mathbb{K}$ of characteristic $0$.

\section{Differential Novikov algebra}

\begin{definition}\label{DerNov}
Let $\Der\Nov_q$ be the quadratic operad defined by the following
identities:
$$(a\prec b)\prec c=(a\prec c)\prec b,$$
$$a\succ(b\succ c)+(b\succ c)\prec a=b\succ(a\succ c)+(a\succ c)\prec b,$$
$$(a\succ b)\succ c-(a\succ c)\prec b=(a\succ c)\succ b-(a\succ b)\prec c$$
and
\begin{multline*}
(a\prec b)\succ c-a\prec(c\prec b)+(a\prec c)\prec b+c\succ(a\prec b)=\\
(c\prec b)\succ a-c\prec(a\prec b)+(c\prec a)\prec b+a\succ(c\prec b).
\end{multline*}
\end{definition}
\noindent These identities arise from the white Manin product of the operad $\Nov$ with itself. In \cite{KolMashSar}, it was proved that
\[
\Nov\circ\Nov=\Nov\otimes\Nov
\]
which gives
\[
\dim(\Nov\circ\Nov(n))=\binom{2n-2}{n-1}^2.
\]
Let $\mathrm{DerNov_q}$ denote the quadratic operad defined by the
identities above. There is a natural epimorphism
\[
\mathrm{DerNov_q}\longrightarrow \Nov\circ\Nov.
\]
This epimorphism is not an isomorphism. Indeed, a computer calculation
using the software~\cite{DotsHij} gives
\[
\dim\mathrm{DerNov_q}(4)=491,
\]
whereas
\[
\dim(\Nov\circ\Nov)(4)=\dim(\Nov\otimes\Nov)(4)=\binom{6}{3}^{2}=400.
\]
Hence the identities above do not form a complete presentation of
$\Nov\circ\Nov$: additional independent identities appear in arity~$4$.
In particular,
\[
\dim\ker\bigl(
\mathrm{DerNov_q}(4)\longrightarrow(\Nov\circ\Nov)(4)
\bigr)=91.
\]
\begin{proposition}
The following identity holds in the operad $\Nov\circ\Nov$:
\begin{align}
(a\succ(b\prec c))\prec d+(a\succ(b\prec d))\prec c=a\succ((b\prec c)\prec d)+((a\succ b)\prec c)\prec d.
\label{degree4-extra}
\end{align}
However, by the software~\cite{DotsHij} identity \eqref{degree4-extra} does not follow from the quadratic defining identities of $\Der\Nov_q$.
\end{proposition}
\begin{proof}
Indeed, using the differential realization
\[
a\prec b=a\,d\partial(b),
\qquad
a\succ b=\partial(a)\,d(b),
\]
where $d$ and $\partial$ are commuting derivations, we obtain
\begin{align*}
(a\succ(b\prec c))\prec d
&=\partial(a)d(b)d\partial(c)d\partial(d)
 +\partial(a)b\,d^2\partial(c)d\partial(d),\\
(a\succ(b\prec d))\prec c
&=\partial(a)d(b)d\partial(c)d\partial(d)
 +\partial(a)b\,d\partial(c)d^2\partial(d),
\end{align*}
whereas
\begin{align*}
a\succ((b\prec c)\prec d)=\partial(a)d(b)d\partial(c)d\partial(d)+\partial(a)b\,d^2\partial(c)d\partial(d)+\partial(a)b\,d\partial(c)d^2\partial(d),
\end{align*}
and
\[
((a\succ b)\prec c)\prec d=\partial(a)d(b)d\partial(c)d\partial(d).
\]
Hence both sides of \eqref{degree4-extra} are equal.
\end{proof}

Set $\Der\Nov^!:=(\Der\Nov_q)^!$. In what follows, we study the quadratic dual $\mathrm{DerNov}^{!}$ of this quadratic operad.

For more details on the white Manin product and Koszul dual operads, see \cite{GinKapr1994}. In \cite{KSO2019}, it was shown that the operad $\Nov\circ\Var:=\Der\Var$ can be embedded into the operad $\Var^d$ via the map $\tau$, where $\Var^d$ denotes the variety $\Var$ equipped with a derivation $d$, and $\tau(a\prec b)=a d(b)$, $\tau(a\succ b)=d(a)b$. If the weight criterion holds, then every algebra in the variety $\Der\Var$ can be embedded into a suitable algebra in $\Var^d$. Here, the weight criterion means that every monomial of $\Var\<X^{(d)}\>$ of weight $-1$ can be expressed in terms of the operations $\succ$ and $\prec$, where the weight function on monomials of $\Var\<X^{(d)}\>$ is defined by
\[
\wt:\Var\<X^{(d)}\>\rightarrow \mathbb{Z},
\]
\[
\wt(x^{(j)}) =j-1,\quad x\in X,\ j\ge 0,
\]
\[
\wt(uv) = \wt(u)+\wt(v)
\]
For example, every noncommutative Novikov algebra $\textrm{N-}\Nov\langle X\rangle$ can be embedded into an appropriate differential associative algebra via the map
\[
\tau:\textrm{N-}\Nov \rightarrow \As^{d}, 
\]
see \cite{erlagol2021}. Since
\[
\Nov\circ\As=\Nov\otimes\As,
\]
we obtain the dimension of the operad $\textrm{N-}\Nov$:
\[
\dim(\textrm{N-}\Nov(n))=n!\binom{2n-2}{n-1}.
\]
However, the equality between the white Manin product and the Hadamard product does not hold for the Zinbiel operad. An example of an algebra in $\Der\Zinb$ that cannot be embedded into any algebra in $\Zinb^d$ was given in \cite{KolMashSar}.

\section[The dual operad of DerNov-q]
{The dual operad of \texorpdfstring{\(\mathrm{DerNov}_{q}\)}{DerNov-q}}

We now compute the quadratic dual of $\Der\Nov_q$. Throughout the rest of the paper, we denote this operad by
\[
\Der\Nov^!:=(\Der\Nov_q)^!.
\]
Put $e_1 = x_1\succ x_2$, $e_2 = x_2\succ x_1$, $g_1=x_1\prec x_2$ and $g_2=x_2\prec x_1$, where $\succ$ and $\prec$ are the two binary operations of $\Der\Nov_q$.

Consider the dual basis
$e_1^\vee, e_2^\vee, g_1^\vee, g_2^\vee$. We write
$e_1^\vee = y_1\succeq y_2$ and
$g_1^\vee = y_1\preceq y_2$; then
$e_2^\vee = - y_2\succeq y_1$ and $g_2^\vee = - y_2\preceq y_1$.
The defining relations among $e_1^\vee, e_2^\vee, g_1^\vee, g_2^\vee$
are precisely those for which the skew-symmetric bracket 
\begin{multline*}
[y_1\otimes x_1, y_2\otimes x_2]
= e_1^\vee \otimes e_1 + e_2^\vee \otimes e_2 + g_1^\vee\otimes g_1 + g_2^\vee \otimes g_2 
=\\
(y_1\succeq y_2)\otimes (x_1\succ x_2)-(y_2\succeq y_1)\otimes (x_2\succ x_1)+
(y_1\preceq y_2)\otimes (x_1\prec x_2)-(y_2\preceq y_1)\otimes (x_2\prec x_1)
\end{multline*}
satisfies the Jacobi identity. More explicitly, we obtain
\begin{multline*}
[[y_1\otimes x_1,y_2\otimes x_2],y_3\otimes x_3]=\\
(y_1\succeq y_2)\succeq y_3\otimes(x_1\succ x_2)\succ x_3-y_3\succeq(y_1\succeq y_2)\otimes x_3\succ (x_1\succ x_2)\\
+(y_1\succeq y_2)\preceq y_3\otimes(x_1\succ x_2)\prec x_3-y_3\preceq(y_1\succeq y_2)\otimes x_3\prec (x_1\succ x_2)\\
-(y_2\succeq y_1)\succeq y_3\otimes(x_2\succ x_1)\succ x_3+y_3\succeq(y_2\succeq y_1)\otimes x_3\succ (x_2\succ x_1)\\
-(y_2\succeq y_1)\preceq y_3\otimes(x_2\succ x_1)\prec x_3+y_3\preceq(y_2\succeq y_1)\otimes x_3\prec (x_2\succ x_1)\\
+(y_1\preceq y_2)\succeq y_3\otimes(x_1\prec x_2)\succ x_3-y_3\succeq(y_1\preceq y_2)\otimes x_3\succ (x_1\prec x_2)\\
+(y_1\preceq y_2)\preceq y_3\otimes(x_1\prec x_2)\prec x_3-y_3\preceq(y_1\preceq y_2)\otimes x_3\prec (x_1\prec x_2)\\
-(y_2\preceq y_1)\succeq y_3\otimes(x_2\prec x_1)\succ x_3+y_3\succeq(y_2\preceq y_1)\otimes x_3\succ (x_2\prec x_1)\\
-(y_2\preceq y_1)\preceq y_3\otimes(x_2\prec x_1)\prec x_3+y_3\preceq(y_2\preceq y_1)\otimes x_3\prec (x_2\prec x_1).
\end{multline*}
Rewriting the second tensor factors in normal form in $\Der\Nov_q$ and collecting the coefficients of the resulting basis monomials, we obtain the defining identities of $\Der\Nov^!$. Using the inclusions (\ref{inclusion}), we may rewrite the second tensor factors in $\Com\langle X^{(\omega,\omega)}\rangle$ as follows:
\begin{multline*}
[[y_1\otimes x_1,y_2\otimes x_2],y_3\otimes x_3]=\\
(y_1\succeq y_2)\succeq y_3\otimes(x_1'' x_2`x_3`+x_1'x_2'`x_3`)-y_3\succeq(y_1\succeq y_2)\otimes (x_3'x_1'`x_2`+x_3'x_1'x_2``)\\
+(y_1\succeq y_2)\preceq y_3\otimes x_1' x_2`x_3'`-y_3\preceq(y_1\succeq y_2)\otimes(x_3x_1''` x_2`+x_3x_1''x_2``+x_3x_1'`x_2'`+x_3x_1'x_2'``)\\
-(y_2\succeq y_1)\succeq y_3\otimes(x_2'' x_1`x_3`+x_2'x_1'`x_3`)+y_3\succeq(y_2\succeq y_1)\otimes (x_3'x_2'`x_1`+x_3'x_2'x_1``)\\
-(y_2\succeq y_1)\preceq y_3\otimes x_2' x_1`x_3'`+y_3\preceq(y_2\succeq y_1)\otimes(x_3x_2''` x_1`+x_3x_2''x_1``+x_3x_2'`x_1'`+x_3x_2'x_1'``)\\
+(y_1\preceq y_2)\succeq y_3\otimes(x_1'x_2'`x_3`+x_1 x_2''`x_3`) 
-y_3\succeq(y_1\preceq y_2)\otimes(x_3'x_1`x_2'`+x_3'x_1x_2'``)\\
+(y_1\preceq y_2)\preceq y_3\otimes x_1x_2'`x_3'`-
y_3\preceq(y_1\preceq y_2)\otimes (x_3x_1'`x_2'`+x_3x_1'x_2'``+x_3x_1`x_2''`+x_3x_1x_2''``)\\
-(y_2\preceq y_1)\succeq y_3\otimes(x_2'x_1'`x_3`+x_2 x_1''`x_3`) 
+y_3\succeq(y_2\preceq y_1)\otimes(x_3'x_2`x_1'`+x_3'x_2x_1'``)\\
-(y_2\preceq y_1)\preceq y_3\otimes x_2x_1'`x_3'`+
y_3\preceq(y_2\preceq y_1)\otimes (x_3x_2'`x_1'`+x_3x_2'x_1'``+x_3x_2`x_1''`+x_3x_2x_1''``),
\end{multline*}
where $x\succ y=x'\circ y$, $x\prec y=x\circ y'$ and $x\circ y=xy`$. 
Similarly,
\begin{multline*}
[[y_2\otimes x_2,y_3\otimes x_3],y_1\otimes x_1]=\\
(y_2\succeq y_3)\succeq y_1\otimes(x_2'' x_3`x_1`+x_2'x_3'`x_1`)-y_1\succeq(y_2\succeq y_3)\otimes (x_1'x_2'`x_3`+x_1'x_2'x_3``)\\
+(y_2\succeq y_3)\preceq y_1\otimes x_2' x_3`x_1'`-y_1\preceq(y_2\succeq y_3)\otimes(x_1x_2''` x_3`+x_1x_2''x_3``+x_1x_2'`x_3'`+x_1x_2'x_3'``)\\
-(y_3\succeq y_2)\succeq y_1\otimes(x_3'' x_2`x_1`+x_3'x_2'`x_1`)+y_1\succeq(y_3\succeq y_2)\otimes (x_1'x_3'`x_2`+x_1'x_3'x_2``)\\
-(y_3\succeq y_2)\preceq y_1\otimes x_3' x_2`x_1'`+y_1\preceq(y_3\succeq y_2)\otimes(x_1x_3''` x_2`+x_1x_3''x_2``+x_1x_3'`x_2'`+x_1x_3'x_2'``)\\
+(y_2\preceq y_3)\succeq y_1\otimes(x_2'x_3'`x_1`+x_2 x_3''`x_1`) 
-y_1\succeq(y_2\preceq y_3)\otimes(x_1'x_2`x_3'`+x_1'x_2x_3'``)\\
+(y_2\preceq y_3)\preceq y_1\otimes x_2x_3'`x_1'`-
y_1\preceq(y_2\preceq y_3)\otimes (x_1x_2'`x_3'`+x_1x_2'x_3'``+x_1x_2`x_3''`+x_1x_2x_3''``)\\
-(y_3\preceq y_2)\succeq y_1\otimes(x_3'x_2'`x_1`+x_3 x_2''`x_1`) 
+y_1\succeq(y_3\preceq y_2)\otimes(x_1'x_3`x_2'`+x_1'x_3x_2'``)\\
-(y_3\preceq y_2)\preceq y_1\otimes x_3x_2'`x_1'`+
y_1\preceq(y_3\preceq y_2)\otimes (x_1x_3'`x_2'`+x_1x_3'x_2'``+x_1x_3`x_2''`+x_1x_3x_2''``)
\end{multline*}
and
\begin{multline*}
[[y_3\otimes x_3,y_1\otimes x_1],y_2\otimes x_2]=\\
(y_3\succeq y_1)\succeq y_2\otimes(x_3'' x_1`x_2`+x_3'x_1'`x_2`)-y_2\succeq(y_3\succeq y_1)\otimes (x_2'x_3'`x_1`+x_2'x_3'x_1``)\\
+(y_3\succeq y_1)\preceq y_2\otimes x_3' x_1`x_2'`-y_2\preceq(y_3\succeq y_1)\otimes(x_2x_3''` x_1`+x_2x_3''x_1``+x_2x_3'`x_1'`+x_2x_3'x_1'``)\\
-(y_1\succeq y_3)\succeq y_2\otimes(x_1'' x_3`x_2`+x_1'x_3'`x_2`)+y_2\succeq(y_1\succeq y_3)\otimes (x_2'x_1'`x_3`+x_2'x_1'x_3``)\\
-(y_1\succeq y_3)\preceq y_2\otimes x_1' x_3`x_2'`+y_2\preceq(y_1\succeq y_3)\otimes(x_2x_1''` x_3`+x_2x_1''x_3``+x_2x_1'`x_3'`+x_2x_1'x_3'``)\\
+(y_3\preceq y_1)\succeq y_2\otimes(x_3'x_1'`x_2`+x_3 x_1''`x_2`) 
-y_2\succeq(y_3\preceq y_1)\otimes(x_2'x_3`x_1'`+x_2'x_3x_1'``)\\
+(y_3\preceq y_1)\preceq y_2\otimes x_3x_1'`x_2'`-
y_2\preceq(y_3\preceq y_1)\otimes (x_2x_3'`x_1'`+x_2x_3'x_1'``+x_2x_3`x_1''`+x_2x_3x_1''``)\\
-(y_1\preceq y_3)\succeq y_2\otimes(x_1'x_3'`x_2`+x_1 x_3''`x_2`) 
+y_2\succeq(y_1\preceq y_3)\otimes(x_2'x_1`x_3'`+x_2'x_1x_3'``)\\
-(y_1\preceq y_3)\preceq y_2\otimes x_1x_3'`x_2'`+
y_2\preceq(y_1\preceq y_3)\otimes (x_2x_1'`x_3'`+x_2x_1'x_3'``+x_2x_1`x_3''`+x_2x_1x_3''``).
\end{multline*}

For simplicity, in $\Der\Nov^!$ we write $\prec$ and $\succ$ instead of $\preceq$ and $\succeq$, respectively. We obtain the following presentation.
\begin{proposition}
The operad $\Der\Nov^!$ is defined by the following identities:
\begin{equation}\label{id11}
a\prec(b\prec c)+(b\prec c)\succ a=0,    
\end{equation}
\begin{equation}\label{id12}
a\prec(b\succ c)=0,  
\end{equation}
\begin{equation}\label{id13}
(a\prec b)\prec c-(c\prec b)\succ a=(a\prec c)\prec b-(b\prec c)\succ a,   
\end{equation}
\begin{equation}\label{id14}
(a\succ b)\prec c+c\succ(a\succ b)-(a\succ c)\succ b=0,
\end{equation}
\begin{equation}\label{id15}
a\succ(b\prec c)+(b\prec c)\succ a=0,
\end{equation}
\begin{equation}\label{id16}
a\succ(b\succ c)=b\succ(a\succ c),
\end{equation}
\begin{equation}\label{id17}
(a\prec b)\succ c=(c\prec b)\succ a
\end{equation}
and
\begin{equation}\label{id18}
(a\succ b)\succ c=(a\succ c)\succ b.
\end{equation}    
\end{proposition}

For later use, we group the identities as follows. Identities (\ref{id11}), (\ref{id12}), (\ref{id14}), and (\ref{id15}) control mixed monomials; identities (\ref{id13}) and (\ref{id17}) give the right Novikov structure of the pure $\prec$-part; and identities (\ref{id16}) and (\ref{id18}) give the bicommutative structure of the pure $\succ$-part.

Let $\Der\Nov^!\<X\>$ denote the free algebra on the countable set $X$ defined by the identities of $\Der\Nov^!$. We show that every monomial in $\Der\Nov^!\<X\>$ can be expressed as a sum of monomials involving only $\prec$ or only $\succ$.

Let $U$ be the linear span of all pure $\prec$-monomials and all pure $\succ$-monomials. We first prove that
\begin{equation}\label{mixed-reduction}
A_{\succ}\prec B_{\prec}\in U
\end{equation}
for arbitrary pure monomials $A_{\succ}$ and $B_{\prec}$. We use induction on $\deg B_{\prec}$. If $\deg B_{\prec}=1$, then the claim is clear when $\deg A_{\succ}=1$. If $A_{\succ}=A'_{\succ}\succ A''_{\succ}$, identity (\ref{id14}) gives
\[
(A'_{\succ}\succ A''_{\succ})\prec B_{\prec}
=(A'_{\succ}\succ B_{\prec})\succ A''_{\succ}
-B_{\prec}\succ(A'_{\succ}\succ A''_{\succ}),
\]
and both monomials on the right are pure $\succ$-monomials.

Assume now that $\deg B_{\prec}>1$ and write
$B_{\prec}=B'_{\prec}\prec B''_{\prec}$. By (\ref{id11}) and (\ref{id17}),
\begin{equation}\label{mixed-induction}
A_{\succ}\prec(B'_{\prec}\prec B''_{\prec})
=-\bigl(A_{\succ}\prec B''_{\prec}\bigr)\succ B'_{\prec}.
\end{equation}
By the induction hypothesis, $A_{\succ}\prec B''_{\prec}$ is a sum of pure monomials. If a summand is a pure $\prec$-monomial of degree at least $2$, then (\ref{id11}) transforms its product $\succ B'_{\prec}$ into a pure $\prec$-monomial. If it has degree $1$, or if the summand is a pure $\succ$-monomial, then either the product is already pure or, when $\deg B'_{\prec}>1$, identities (\ref{id15}) and (\ref{id11}) reduce it again to a product of the form \eqref{mixed-reduction} with a strictly smaller $\prec$-degree. Thus \eqref{mixed-reduction} follows by induction.

The remaining mixed products are handled similarly. If $A_{\prec}$ and $B_{\succ}$ are pure, then
\[
A_{\prec}\prec B_{\succ}=0
\]
whenever $\deg B_{\succ}>1$, by (\ref{id12}); for $\deg B_{\succ}=1$ the product is pure $\prec$. If $\deg A_{\prec}>1$, then (\ref{id11}) gives
\[
A_{\prec}\succ B_{\succ}=-B_{\succ}\prec A_{\prec}\in U.
\]
Finally, if $\deg B_{\prec}>1$, then (\ref{id15}) and (\ref{id11}) give
\[
A_{\succ}\succ B_{\prec}=A_{\succ}\prec B_{\prec}\in U.
\]
Hence $U$ is closed under both operations and contains $X$, so $U=\Der\Nov^!\<X\>$. Notice that the only recursive step above is \eqref{mixed-induction}, where the degree of the pure $\prec$-factor strictly decreases. Therefore, the reduction process terminates.

Moreover, identities (\ref{id11}) and (\ref{id17}) imply
\[
a\prec(b\prec c)=b\prec(a\prec c),
\]
while identities (\ref{id11}) and (\ref{id13}) imply
\[
(a\prec b)\prec c-a\prec(b\prec c)
=(a\prec c)\prec b-a\prec(c\prec b).
\]
Thus the pure $\prec$-part is a right Novikov algebra. Identities (\ref{id16}) and (\ref{id18}) show that the pure $\succ$-part is bicommutative.

We therefore obtain the following decomposition:

\begin{proposition}
The vector space $\Der\Nov^!\langle X\rangle$ has the decomposition
\[
\Der\Nov^!\langle X\rangle
=
\Der\Nov^!_{\prec}\langle X\rangle
+
\Der\Nov^!_{\succ}\langle X\rangle,
\]
where $\Der\Nov^!_{\prec}\langle X\rangle$ and $\Der\Nov^!_{\succ}\langle X\rangle$ denote the subspaces of $\Der\Nov^!\langle X\rangle$ generated by monomials involving only the operations $\prec$ and $\succ$, respectively.

Moreover, $\Der\Nov^!_{\prec}\langle X\rangle$ is a right Novikov algebra with additional identities of higher degree, whereas $\Der\Nov^!_{\succ}\langle X\rangle$ is a bicommutative algebra with additional identities of higher degree.
\end{proposition}

The sum in the proposition is not direct in degree $1$, since $KX$, the linear span of $X$ over the ground field $K$, belongs to both pure components. We will prove after Theorem~\ref{th1} that the sum is direct in every degree at least $2$.

We next determine the complete sets of identities satisfied by the algebras
\[
\Der\Nov^!_{\prec}\langle X\rangle
\quad \text{and} \quad
\Der\Nov^!_{\succ}\langle X\rangle.
\]

\begin{theorem}\label{th1}
\leavevmode

\smallskip

\noindent
\textup{(1)} The algebra $\Der\Nov^!_{\prec}\langle X\rangle$ is the free right Novikov algebra subject to the following additional identities:
\begin{equation}\label{Nov1}
(a\prec(b\prec c))\prec d
=
a\prec(b\prec(c\prec d))
=
0,
\end{equation}
\begin{equation}\label{Nov2}
a\prec((b\prec c)\prec d)
=
a\prec((b\prec d)\prec c),
\end{equation}
and
\begin{equation}\label{Nov3}
((a\prec b)\prec c)\prec d
=
((a\prec c)\prec b)\prec d.
\end{equation}

\smallskip

\noindent
\textup{(2)} The algebra $\Der\Nov^!_{\succ}\langle X\rangle$ is the free bicommutative algebra subject to the following additional identities:
\begin{equation}\label{bicom1}
((a\succ b)\succ c)\succ d
=
d\succ(c\succ(b\succ a))
=
c\succ(b\succ(a\succ d)),
\end{equation}
\begin{equation}\label{bicom2}
((a\succ b)\succ c)\succ d
=
c\succ((d\succ a)\succ b)
=
b\succ((d\succ a)\succ c)
=
b\succ((c\succ a)\succ d).
\end{equation}
\end{theorem}

\begin{proof}
First, we verify identities (\ref{Nov1})--(\ref{Nov3}) in $\Der\Nov^!\<X\>$. From (\ref{id11}) and (\ref{id15}), we obtain
\begin{equation}\label{mixed1}
a\succ(b\prec c)=a\prec(b\prec c).
\end{equation}
Also, by (\ref{id12}) and (\ref{id11}),
\[
a\prec((c\prec d)\succ b)=0,
\qquad
(c\prec d)\succ b=-b\prec(c\prec d),
\]
therefore
\begin{equation}\label{Nov1b}
a\prec(b\prec(c\prec d))=0.
\end{equation}

From (\ref{id11}) and (\ref{id13}), we have
\begin{equation}\label{mixed2}
(b\prec a)\succ c
=(c\prec a)\prec b-(c\prec b)\prec a-a\prec(c\prec b).
\end{equation}
Replacing $c$ by $c\prec d$ in \eqref{mixed2} and using the right Novikov identities, we obtain
\[
(b\prec a)\succ(c\prec d)=-(c\prec d)\prec(b\prec a).
\]
On the other hand, by \eqref{mixed1} and the right-symmetric identity,
\begin{multline*}
(b\prec a)\succ(c\prec d)
=(b\prec a)\prec(c\prec d)\\
=b\prec(a\prec(c\prec d))
+(b\prec(c\prec d))\prec a
-b\prec((c\prec d)\prec a).
\end{multline*}
The first term is zero by \eqref{Nov1b}, and left commutativity gives
\[
b\prec((c\prec d)\prec a)=(c\prec d)\prec(b\prec a).
\]
Comparing the two expressions, we obtain
\[
(b\prec(c\prec d))\prec a=0.
\]
Together with \eqref{Nov1b}, this proves (\ref{Nov1}).

The remaining identities follow immediately from the right-symmetric identity. Indeed, applying $a\prec(-)$ to
\[
(b\prec c)\prec d-b\prec(c\prec d)
=(b\prec d)\prec c-b\prec(d\prec c)
\]
and using (\ref{Nov1}), we obtain (\ref{Nov2}). Multiplying the right-symmetric identity on the right by $d$ and again using (\ref{Nov1}), we obtain (\ref{Nov3}).

We shall use the following consequence several times. By the right-symmetric identity and (\ref{Nov1}),
\begin{equation}\label{Nov4}
(a\prec b)\prec(c\prec d)
=-a\prec((c\prec d)\prec b).
\end{equation}
Indeed,
\begin{multline*}
(a\prec b)\prec(c\prec d)-a\prec(b\prec(c\prec d))=(a\prec(c\prec d))\prec b-a\prec((c\prec d)\prec b),
\end{multline*}
and the two middle terms are zero by (\ref{Nov1}).

Next, we verify identities (\ref{bicom1}) and (\ref{bicom2}) in $\Der\Nov^!\<X\>$. From (\ref{id11}) and (\ref{id15}), we have
\[
b\succ(a\prec c)=a\succ(b\prec c).
\]
Using this identity and (\ref{id14}),
\[
(b\succ d)\succ(a\prec c)=a\succ((b\succ d)\prec c)=-a\succ(c\succ(b\succ d))+a\succ((b\succ c)\succ d).
\]
On the other hand, by (\ref{id18}), (\ref{id15}), and (\ref{id11}),
\begin{align*}
(b\succ d)\succ(a\prec c)=(b\succ(a\prec c))\succ d=-((a\prec c)\succ b)\succ d\\
=(b\prec(a\prec c))\succ d=-d\prec(b\prec(a\prec c))=0
\end{align*}
by (\ref{Nov1}). Hence
\begin{equation}\label{help1}
a\succ(c\succ(b\succ d))=a\succ((b\succ c)\succ d).
\end{equation}

Using (\ref{id16}), \eqref{help1}, and (\ref{id18}), we obtain
\begin{align*}
a\succ(b\succ(c\succ d))=a\succ(c\succ(b\succ d))=a\succ((b\succ c)\succ d)\\
=a\succ((b\succ d)\succ c)=a\succ(d\succ(b\succ c))=a\succ(b\succ(d\succ c)).
\end{align*}
Therefore,
\begin{equation}\label{help2}
a\succ(b\succ(d\succ c))=a\succ(b\succ(c\succ d)).
\end{equation}

Finally, from the composition of (\ref{id14}) and (\ref{id15}),
\[
((a\succ d)\prec b)\succ c
=((a\succ d)\succ b)\succ c-(a\succ(b\succ d))\succ c.
\]
Indeed, setting $P=(a\succ d)\prec b$, identities (\ref{id15}) and (\ref{id14}) give
\[
P\succ c=-c\succ P
=c\succ(b\succ(a\succ d))-c\succ((a\succ b)\succ d)=0
\]
by \eqref{help1}. Hence
\begin{equation}\label{help3}
((a\succ d)\succ b)\succ c=(a\succ(b\succ d))\succ c.
\end{equation}
Identities (\ref{bicom1}) and (\ref{bicom2}) now follow from \eqref{help1}, \eqref{help2}, \eqref{help3}, and bicommutativity.

It remains to prove that there are no further identities involving only one of the two operations. We first consider the operation $\prec$. Let $P$ be the free right Novikov algebra satisfying (\ref{Nov1})--(\ref{Nov3}), and write
\[
P=\bigoplus_{n\geq1}P_n,
\qquad
P_{\geq2}=\bigoplus_{n\geq2}P_n.
\]
We retain the notation $u\prec v$ for the product in $P$. Since every element of $P_{\geq2}$ is a linear combination of products, identity (\ref{Nov1}) implies
\begin{equation}\label{annihilator-relations}
(a\prec v)\prec c=0,
\qquad
a\prec(c\prec v)=0
\end{equation}
for $v\in P_{\geq2}$.

For $u,v\in P_{\geq2}$, identity \eqref{Nov4} and left commutativity give
\begin{equation}\label{anti-products}
u\prec v=-v\prec u.
\end{equation}
Indeed, by multilinearity, it suffices to take $u=a\prec b$ and $v=c\prec d$; then
\[
(a\prec b)\prec(c\prec d)
=-a\prec((c\prec d)\prec b)
=-(c\prec d)\prec(a\prec b).
\]
Put
\[
J=P\prec P_{\geq2}.
\]
By \eqref{annihilator-relations},
\begin{equation}\label{annihilator}
J\prec P=P\prec J=0.
\end{equation}

We now define a second operation on $P$. For homogeneous $u,v\in P$, set
\begin{equation}\label{succ-on-P}
u\succ v=
\begin{cases}
0, & u,v\in P_1,\\
u\prec v, & v\in P_{\geq2},\\
-v\prec u, & u\in P_{\geq2}.
\end{cases}
\end{equation}
and extend it bilinearly. If $u,v\in P_{\geq2}$, the last two formulas agree by \eqref{anti-products}. Moreover, every $\succ$-product belongs to $J$.

We claim that $(P,\prec,\succ)$ satisfies (\ref{id11})--(\ref{id18}). Identities (\ref{id11}) and (\ref{id15}) follow directly from \eqref{succ-on-P}. Identity (\ref{id12}) follows from \eqref{annihilator}. Identity (\ref{id17}) follows from left commutativity, while (\ref{id13}) becomes precisely the right-symmetric identity. Finally, all terms in (\ref{id14}), (\ref{id16}), and (\ref{id18}) vanish by \eqref{annihilator}, since every $\succ$-product lies in $J$. Thus $P$ is a $\Der\Nov^!$-algebra.

Let $F=\Der\Nov^!\<X\>$ and let $F_{\prec}$ be its pure $\prec$-part. There is a natural epimorphism
\[
P\longrightarrow F_{\prec}.
\]
Conversely, the construction above gives a homomorphism $F\to P$ that is the identity on the generators. Its restriction to $F_{\prec}$ is a left inverse of the previous epimorphism. Hence
\begin{equation}\label{prec-iso}
F_{\prec}\cong P.
\end{equation}
This proves that the right Novikov identities together with (\ref{Nov1})--(\ref{Nov3}) form the complete list of identities of the pure $\prec$-part.

Now let $Q$ be the free bicommutative algebra satisfying (\ref{bicom1}) and (\ref{bicom2}). These identities and bicommutativity imply, in degree $4$,
\begin{equation}\label{bicom-assoc}
((a\succ b)\succ c)\succ d
=(a\succ(b\succ c))\succ d
=(a\succ b)\succ(c\succ d)
=a\succ(b\succ(c\succ d)),
\end{equation}
and also
\begin{equation}\label{bicom-sym}
((a\succ b)\succ c)\succ d
=((b\succ a)\succ c)\succ d.
\end{equation}
In fact, \eqref{bicom-sym} follows from \eqref{bicom-assoc} and bicommutativity. By applying \eqref{bicom-assoc} to submonomials and using left- and right-commutativity, every pure $\succ$-monomial of degree at least $4$ can be reduced to a left-normed monomial, and all its variables can be ordered. Consequently, every fixed multihomogeneous component of $Q$ of degree at least $4$ has dimension at most $1$.

Consider the one-dimensional algebra $Ke$ with
\[
e\prec e=0,
\qquad
e\succ e=e.
\]
It satisfies (\ref{id11})--(\ref{id18}). Every pure $\succ$-monomial evaluates to $e$, so no canonical pure $\succ$-monomial of degree at least $4$ can vanish in $F$. In degree $3$, the quadratic calculation leading to Proposition~1 gives
\[
\dim F(3)=12,
\]
while the degree-$3$ components of the right Novikov and bicommutative operads both have dimension $6$. Together with Proposition~2 and \eqref{prec-iso}, this excludes any additional pure $\succ$-relation in degree $3$. Hence
\begin{equation}\label{succ-iso}
F_{\succ}\cong Q.
\end{equation}
This proves that (\ref{bicom1}) and (\ref{bicom2}) form the complete set of additional identities for the pure $\succ$-part; degree $2$ is immediate since all defining relations have arity $3$.

We record two sample compositions. For
\[
w=(a\prec(b\succ c))\succ d,
\]
reducing the inner monomial by (\ref{id12}) gives $w=0$, while reducing the outer monomial by (\ref{id11}) gives
\[
w=-d\prec(a\prec(b\succ c))=0
\]
again by (\ref{id12}). Next, for
\[
w=(a\prec b)\succ(c\prec d),
\]
identities (\ref{id15}) and (\ref{id11}) give
\[
w=(a\prec b)\prec(c\prec d),
\]
whereas (\ref{id11}) gives
\[
w=-(c\prec d)\prec(a\prec b).
\]
The two results coincide by \eqref{Nov4} and left commutativity.

We also record the reduction requested in the referee report. Omitting the symbol $\prec$ in this calculation, we have
\[
\begin{aligned}
x_2(((x_1x_3)x_4)x_5)=((x_1x_3)x_4)(x_2x_5)=-(x_1x_3)((x_2x_5)x_4)=x_1(((x_2x_5)x_4)x_3)\\
=x_1(((x_2x_4)x_5)x_3)=x_1(((x_2x_4)x_3)x_5)=x_1(((x_2x_3)x_4)x_5).
\end{aligned}
\]
Here the first equality follows from left commutativity, the next two from \eqref{Nov4}, and the last three from (\ref{Nov3}), (\ref{Nov2}), and (\ref{Nov3}), respectively. Thus the coefficient $k$ in the notation of the referee is equal to $1$.

For completeness, we describe the computer check used in the calculation. We work with the associated unsigned shuffle operad. In the notation of the software~\cite{DotsHij}, the monomial order is
\[
\texttt{measure: deglex perm},
\]
with the ordered signature
\[
\texttt{x(2)\ y(2)\ t(2)\ z(2)},
\]
where $x(1,2)$ and $y(1,2)$ are the two shuffle representatives of $\prec$, while $t(1,2)$ and $z(1,2)$ are the two shuffle representatives of $\succ$. Thus tree monomials are first compared by the degree-lexicographic path order and, if the path data coincide, by the permutation of the leaves. For example,
\[
x(1,2)<z(1,2),
\qquad
x(t(1,2),3)<t(x(1,2),3),
\]
which correspond to
\[
a\prec b<b\succ a,
\qquad
(a\succ b)\prec c<(a\prec b)\succ c.
\]
The partial Buchberger computation through arity $8$ gives
\[
12,\ 11,\ 16,\ 22,\ 29,\ 37
\]
normal monomials in arities $3,4,5,6,7,8$. Starting from arity $5$, further rewriting rules involving only $\prec$ appear. They are consequences of the right Novikov identities and (\ref{Nov1})--(\ref{Nov3}), as proved above, and therefore do not give new defining identities.

Finally, the homomorphism $F\to P$ constructed above sends every pure $\succ$-monomial of degree at least $2$ to zero and is injective on $F_{\prec}$. Hence
\[
F_{\prec}\cap F_{\succ}=0
\qquad\text{in every degree at least }2.
\]
Therefore the decomposition of Proposition~2 refines to
\begin{equation}\label{direct-decomposition}
\Der\Nov^!\<X\>
=KX\oplus F_{\prec}^{\geq2}\oplus F_{\succ}^{\geq2}.
\end{equation}
\end{proof}

\section{A basis of the algebra \texorpdfstring{$\Der\Nov^!\langle X\rangle$}{DerNov!<X>}}

By \eqref{direct-decomposition}, to construct a basis of the algebra $\Der\Nov^!\<X\>$ it suffices to construct bases for the pure $\prec$- and $\succ$-components separately. 

Define the sets $\mathcal{N}_i$ as follows:
\[
\mathcal{N}_1=\{x_i\},\;\mathcal{N}_2=\{x_ix_j\},\;\mathcal{N}_3=\{(x_{k_1}x_{k_2})x_{k_3},x_{l_1}(x_{l_2}x_{l_3})\}
\]
and
\[
\mathcal{N}_n=\{((\cdots(x_{m_1}x_{m_{2}})\cdots)x_{m_{n-1}})x_{m_n}, x_{r_1}((\cdots(x_{r_2}x_{r_{3}})\cdots)x_{r_n})\},
\]
where $k_2\leq k_3$, $l_1\geq l_2$, $n\geq 4$, $m_n\geq \ldots \geq m_2$, $r_n\geq\ldots\geq r_3$, and
\[
r_2\geq r_1\quad\text{if $n$ is odd},
\qquad
r_2>r_1\quad\text{if $n$ is even}.
\]
Set
\[
\mathcal{N}=\bigcup_{i}\mathcal{N}_i.
\]
\begin{theorem}
Let $\Nov_s\<X\>$ denote the free right Novikov algebra satisfying identities (\ref{Nov1}), (\ref{Nov2}), and (\ref{Nov3}). The set $\mathcal{N}$ is a basis of $\Nov_s\<X\>$.
\end{theorem}
\begin{proof}
For simplicity, we omit the symbol $\prec$ in this proof. Put
\[
L_n(a_1,\ldots,a_n)=((\cdots(a_1a_2)a_3)\cdots)a_n
\]
and
\[
R_n(a,b;c_1,\ldots,c_{n-2})
=aL_{n-1}(b,c_1,\ldots,c_{n-2}).
\]
Identity \eqref{Nov4} implies, by induction on $n$, that for $n,m\geq2$,
\begin{equation}\label{LL-product}
L_n(a_1,\ldots,a_n)L_m(b_1,\ldots,b_m)
=(-1)^{n-1}R_{n+m}(a_1,b_1;b_2,\ldots,b_m,a_n,\ldots,a_2).
\end{equation}
Indeed, each application of \eqref{Nov4} moves one letter of the left factor to the end of the inner left-normed word and changes the sign.

Identity (\ref{Nov1}) shows that every monomial of the form $R_n$, $n\geq3$, annihilates the algebra on both sides. Consequently, an induction on the degree shows that every monomial of $\Nov_s\<X\>$ is a linear combination of monomials of the two forms $L_n$ and $R_n$.

We next order the variables. Identity (\ref{Nov3}) allows us to interchange adjacent variables inside $L_n$ as long as a factor remains on the right. The last two variables can be interchanged by the right-symmetric identity; the correction terms are products of two left-normed monomials and hence are of type $R_n$ by \eqref{LL-product}. Thus $a_2,\ldots,a_n$ can be ordered modulo $R_n$-terms.

For $R_n$, identities (\ref{Nov2}) and (\ref{Nov3}) show that the variables $c_1,\ldots,c_{n-2}$ can be arbitrarily permuted. Moreover, left commutativity and \eqref{LL-product} give
\begin{equation}\label{R-symmetry}
R_n(a,b;c_1,\ldots,c_{n-2})
=(-1)^{n-3}R_n(b,a;c_1,\ldots,c_{n-2}).
\end{equation}
For $n=3$, this is exactly left commutativity. For $n\geq4$, we have
\begin{align*}
R_n(a,b;c_1,\ldots,c_{n-2})
&=L_{n-2}(b,c_1,\ldots,c_{n-3})(a c_{n-2})\\
&=(-1)^{n-3}R_n(b,a;c_{n-2},c_{n-3},\ldots,c_1),
\end{align*}
and the last variables can be reordered. Hence the first two variables can also be ordered. If $n$ is even and $a=b$, \eqref{R-symmetry} gives $R_n(a,a;\ldots)=0$. This explains the strict inequality in the definition of $\mathcal N_n$ for even $n$. Thus $\mathcal N$ spans $\Nov_s\<X\>$.

It remains to prove linear independence. Consider the vector space $A\<X\>$ with basis $\mathcal N$. We define the multiplication on basis elements as follows. Every product involving an $R$-type basis monomial of degree at least $3$ is zero, and
\[
x_i x_j=L_2(x_i,x_j).
\]
For $m\geq2$, put
\[
x_iL_m(j_1,\ldots,j_m)=R_{m+1}(i,j_1;j_2,\ldots,j_m),
\]
and then rewrite the resulting $R$-monomial in canonical order. For $n,m\geq2$, the product of two left-normed basis monomials is defined by \eqref{LL-product}, again followed by the canonical ordering. More explicitly, after ordering the tail variables, we have
\begin{equation}\label{correct-multiplication}
L_n(i_1,\ldots,i_n)L_m(j_1,\ldots,j_m)
=\varepsilon\,R_{n+m}(k_1,k_2;k_3,\ldots,k_{n+m}),
\end{equation}
where $k_3\leq\cdots\leq k_{n+m}$ is the ordered list of
\[
i_2,\ldots,i_n,j_2,\ldots,j_m,
\]
and $\{k_1,k_2\}=\{i_1,j_1\}$ with $k_1\leq k_2$. The coefficient is
\[
\varepsilon=
\begin{cases}
(-1)^{n-1}, & i_1<j_1,\\
(-1)^m, & i_1>j_1,\\
(-1)^{n-1}, & i_1=j_1\text{ and }n+m\text{ is odd},\\
0, & i_1=j_1\text{ and }n+m\text{ is even}.
\end{cases}
\]

It remains to define the product of a left-normed basis monomial by a generator. Assume that
\[
i_2\leq\cdots\leq i_n.
\]
If $j\geq i_n$, put
\[
L_n(i_1,\ldots,i_n)x_j=L_{n+1}(i_1,i_2,\ldots,i_n,j).
\]
If $j<i_n$, the right-symmetric identity gives
\begin{align}
L_n(i_1,\ldots,i_n)x_j
&=L_{n+1}\bigl(i_1,\operatorname{ord}(i_2,\ldots,i_n,j)\bigr)\notag\\
&\quad+(-1)^{n-2}R_{n+1}(i_1,i_n;j,i_{n-1},\ldots,i_2)\notag\\
&\quad-(-1)^{n-2}R_{n+1}(i_1,j;i_n,i_{n-1},\ldots,i_2),
\label{L-generator-product}
\end{align}
where $\operatorname{ord}$ means that the tail variables are arranged in nondecreasing order, and the two $R$-terms are then rewritten in canonical form using \eqref{R-symmetry}. For $n=2$, the strings $i_{n-1},\ldots,i_2$ in \eqref{L-generator-product} are empty, and the formula follows directly from the right-symmetric identity. For $n\geq3$, formula \eqref{L-generator-product} is the right-symmetric identity
\[
(uv)w=(uw)v+u(vw)-u(wv),
\]
followed by \eqref{LL-product}; after the first interchange all further movements of $x_j$ are performed by (\ref{Nov3}) and produce no additional terms.

These rules define the product of any pair of basis elements. A direct case-by-case verification for generators, $L$-type monomials, and $R$-type monomials, together with \eqref{LL-product}, \eqref{R-symmetry}, and the fact that all $R$-type monomials annihilate $A\<X\>$, shows that $A\<X\>$ satisfies the right Novikov identities and (\ref{Nov1})--(\ref{Nov3}). Hence there is a natural epimorphism
\[
\Nov_s\<X\>\longrightarrow A\<X\>.
\]
Since $\mathcal N$ spans $\Nov_s\<X\>$ and is a basis of $A\<X\>$ by construction, this map is an isomorphism. Therefore $\mathcal N$ is a basis of $\Nov_s\<X\>$.
\end{proof}

Define the sets $\mathcal{B}_i$ as follows:
\[
\mathcal{B}_1=\{x_i\},\;\mathcal{B}_2=\{x_ix_j\},\;\mathcal{B}_3=\{(x_{k_1}x_{k_2})x_{k_3},x_{l_1}(x_{l_2}x_{l_3})\}
\]
and
\[
\mathcal{B}_n=\{((\cdots(x_{m_1}x_{m_{2}})\cdots)x_{m_{n-1}})x_{m_n}\},
\]
where $k_2\leq k_3$, $l_1\geq l_2$, $n\geq 4$ and $m_n\geq \ldots \geq m_2\geq m_1$. Set
\[
\mathcal{B}=\bigcup_{i}\mathcal{B}_i.
\]
\begin{theorem}
Let $\Bi\Com_s\<X\>$ denote the free bicommutative algebra satisfying identities (\ref{bicom1}) and (\ref{bicom2}). The set $\mathcal{B}$ is a basis of $\Bi\Com_s\<X\>$.
\end{theorem}
\begin{proof}
For monomials of degree at most $3$, the claim is the standard basis description for a bicommutative algebra. Assume that the degree is at least $4$. Identities (\ref{bicom1}) and (\ref{bicom2}), together with bicommutativity, imply
\begin{equation}\label{bicom-normal-assoc}
((ab)c)d=(a(bc))d=(ab)(cd)=a(b(cd))
\end{equation}
and
\begin{equation}\label{bicom-normal-sym}
((ab)c)d=((ba)c)d.
\end{equation}
Here and below the product denotes $\succ$. The second identity also follows from the first one and bicommutativity.

Applying \eqref{bicom-normal-assoc} to submonomials, every monomial of degree at least $4$ can be written as a left-normed monomial. Right commutativity allows us to order all variables except possibly the first one, while \eqref{bicom-normal-sym} interchanges the first two variables. Hence every monomial is equal to a canonical monomial
\[
((\cdots(x_{m_1}x_{m_2})\cdots)x_{m_{n-1}})x_{m_n},
\qquad
m_1\leq m_2\leq\cdots\leq m_n.
\]
Thus $\mathcal B$ spans $\Bi\Com_s\<X\>$.

For linear independence, consider the vector space $A\<X\>$ with basis $\mathcal B$. For products of total degree at most $3$, define the multiplication as in the free bicommutative algebra. For any two basis monomials whose total degree is at least $4$, define their product to be the canonical left-normed monomial obtained by collecting and ordering all variables occurring in the two factors. Then \eqref{bicom-normal-assoc} and \eqref{bicom-normal-sym} show directly that $A\<X\>$ satisfies bicommutativity and identities (\ref{bicom1}), (\ref{bicom2}). Therefore the natural epimorphism
\[
\Bi\Com_s\<X\>\longrightarrow A\<X\>
\]
is an isomorphism, and $\mathcal B$ is a basis of $\Bi\Com_s\<X\>$.
\end{proof}

\begin{proposition}
The operad $\Bi\Com$ satisfies the Dong property in the sense of \cite{Dongprop}.
\end{proposition}
\begin{proof}
Let $U$ be a $\Bi\Com$-algebra. The Lie-admissibility condition for $S\otimes U$ yields the defining identities of the operad $\Bi\Com^!$. We have
\begin{multline*}
[[a\otimes u,b\otimes v],c\otimes w]=(ab)c\otimes (uv)w-(ba)c\otimes (vu)w-
c(ab)\otimes w(uv)+c(ba)\otimes w(vu),
\end{multline*}
\begin{multline*}
[[b\otimes v,c\otimes w],a\otimes u]=(bc)a\otimes (vw)u-(cb)a\otimes (wv)u-
a(bc)\otimes u(vw)+a(cb)\otimes u(wv),
\end{multline*}
and
\begin{multline*}
[[c\otimes w,a\otimes u],b\otimes v]=(ca)b\otimes (wu)v-(ac)b\otimes (uw)v-
b(ca)\otimes v(wu)+b(ac)\otimes v(uw).
\end{multline*}
Summing these expressions and collecting the coefficients of the same basis monomials in the second tensor factors, we obtain
\[
(ab)c=(ac)b\;\;\textrm{and}\;\; a(bc)=b(ac).
\]
Thus the operad $\Bi\Com$ is self-dual, and the monomials $(ab)c$, $(ba)c$, $c(ab)$, and $c(ba)$ are linearly independent in $\Bi\Com^!\<X\>$. The criterion of \cite{Dongprop} therefore implies that $\Bi\Com$ satisfies the Dong property.
\end{proof}

\begin{corollary}
The operad $\Der\Nov^!$ satisfies the Dong property.
\end{corollary}
\begin{proof}
By the criterion of \cite{Dongprop}, it suffices to show that
every degree-$3$ monomial in which $x_3$ occurs in the outer
position can be expressed as a linear combination of monomials
in which $x_3$ occurs in an inner position.

For monomials involving only $\prec$, the claim follows from
the Dong property of the right Novikov operad. For monomials
involving only $\succ$, it follows from the Dong property of
the bicommutative operad. We also use the fact that the Dong
criterion is invariant under passage to the opposite operad.

It remains to consider mixed monomials. Up to interchanging
$x_1$ and $x_2$, we have
\[
(x_1\prec x_2)\succ x_3
=(x_3\prec x_2)\succ x_1
\]
by (\ref{id17}),
\[
x_3\succ(x_1\prec x_2)
=-(x_1\prec x_2)\succ x_3
=-(x_3\prec x_2)\succ x_1
\]
by (\ref{id15}) and (\ref{id17}),
\[
\begin{aligned}
(x_1\succ x_2)\prec x_3
&=-x_3\succ(x_1\succ x_2)
 +(x_1\succ x_3)\succ x_2\\
&=-x_1\succ(x_3\succ x_2)
 +(x_1\succ x_3)\succ x_2
\end{aligned}
\]
by (\ref{id14}) and (\ref{id16}), and
\[
x_3\prec(x_1\succ x_2)=0
\]
by (\ref{id12}).

Thus every mixed monomial with $x_3$ in the outer position is a linear
combination of monomials with $x_3$ in an inner position. Hence all
sixteen compositions satisfy the criterion of \cite{Dongprop},
and $\Der\Nov^!$ has the Dong property.
\end{proof}

Finally, the low-arity dimensions are
\[
\dim(\Der\Nov^!(1))=1,\qquad \dim(\Der\Nov^!(2))=4,\qquad \dim(\Der\Nov^!(3))=12,
\]
and for $n\geq4$ we have
\[
\dim(\Der\Nov^!(n))=\frac{n(n+1)}{2}+1.
\]

\subsection*{Acknowledgments}
This research was funded by the Science Committee of the Ministry of Science and Higher Education of the Republic of Kazakhstan (Grant No. AP32517747).

\end{document}